\documentclass[12pt]{elsarticle}

\usepackage{amssymb}
\usepackage{amsmath}
\usepackage{amsthm}
\usepackage{wrapfig}
\usepackage{graphicx}

\newtheorem{prop}{Proposition}
\newtheorem{corr}{Corollary}

\newdefinition{rem}{Remark}
\newdefinition{defi}{Definition}

\newcommand{\mS}{\mathcal{S}}

\newcommand{\ex}{\textnormal{exp}}
\newcommand{\beq}{\begin{equation}}
\newcommand{\eeq}{\end{equation}}
\newcommand{\dd}{\partial}
\newcommand{\gd}{\dot{\gamma}}

\journal{*****}

\begin{document}

\begin{frontmatter}

\title{Bifurcations of the conjugate locus}

\author{Thomas Waters\fnref{label2}}
\ead{thomas.waters@port.ac.uk}
\address{Department of Mathematics, University of Portsmouth, England PO13HF}

\begin{abstract}
The conjugate locus of a point $p$ in a surface $\mS$ will have a certain number of cusps. As the point $p$ is moved in the surface the conjugate locus may spontaneously gain or lose cusps. In this paper we explain this `bifurcation' in terms of the vanishing of higher derivatives of the exponential map; we derive simple equations for these higher derivatives in terms of scalar invariants; we classify the bifurcations of cusps in terms of the local structure of the conjugate locus; and we describe an intuitive picture of the bifurcation as the intersection between certain contours in the tangent plane.
\end{abstract}


\end{frontmatter}

\section{Introduction}

The conjugate locus, and its relative the cut locus, are classical objects in Differential Geometry and have been studied deeply by many mathematicians since the middle of the 19th century (some important works are \cite{jacobi}, \cite{poincare},\cite{myers},\cite{whitehead}). Of particular relevance to this paper is the so-called ``last geometric statement of Jacobi'', which asserts (among other things) that the conjugate locus of a non-umbilic point on the triaxial ellipsoid has precisely 4 cusps (see \cite{Sinclair1} for a historical sketch and list of references). This conjecture was recently proved by Itoh and Kiyohara \cite{Itoh1}, and a renewed interest in the conjugate and cut locus can be seen in the recent papers providing formal studies (\cite{grav}, \cite{Sinclair3}, \cite{Itoh3}, \cite{Itoh2}), simulations (\cite{Sinclair2}, \cite{Bonnard2}, \cite{Sinclair5}, \cite{Sinclair4}, \cite{Sinclair1}) and applications (\cite{Bonnard1}, \cite{Bonnard3}, \cite{bloch}, \cite{Bonnard4}, \cite{Bonnard5}).

It is no surprise that the papers which focused on the triaxial ellipsoid and surfaces of revolution made heavy use of the fact that the geodesic flow on those surfaces is (Liouville) integrable. However, surfaces for which a second integral exists are few and far between, and some recent papers of the author (\cite{TWspherical},\cite{TWmonge},\cite{TWetds}) showed that for even simple surfaces the geodesic flow may be non-integrable and indeed chaotic. We could say that surfaces with integrable geodesic flow are exceptional, and an underlying goal of this paper is to develop techniques to understand the fine structure of the geodesic flow {\it without} relying on additional integrals.

\begin{figure}
\begin{picture}(200,170)
\put(0,0){{\includegraphics[width=0.25\textwidth]{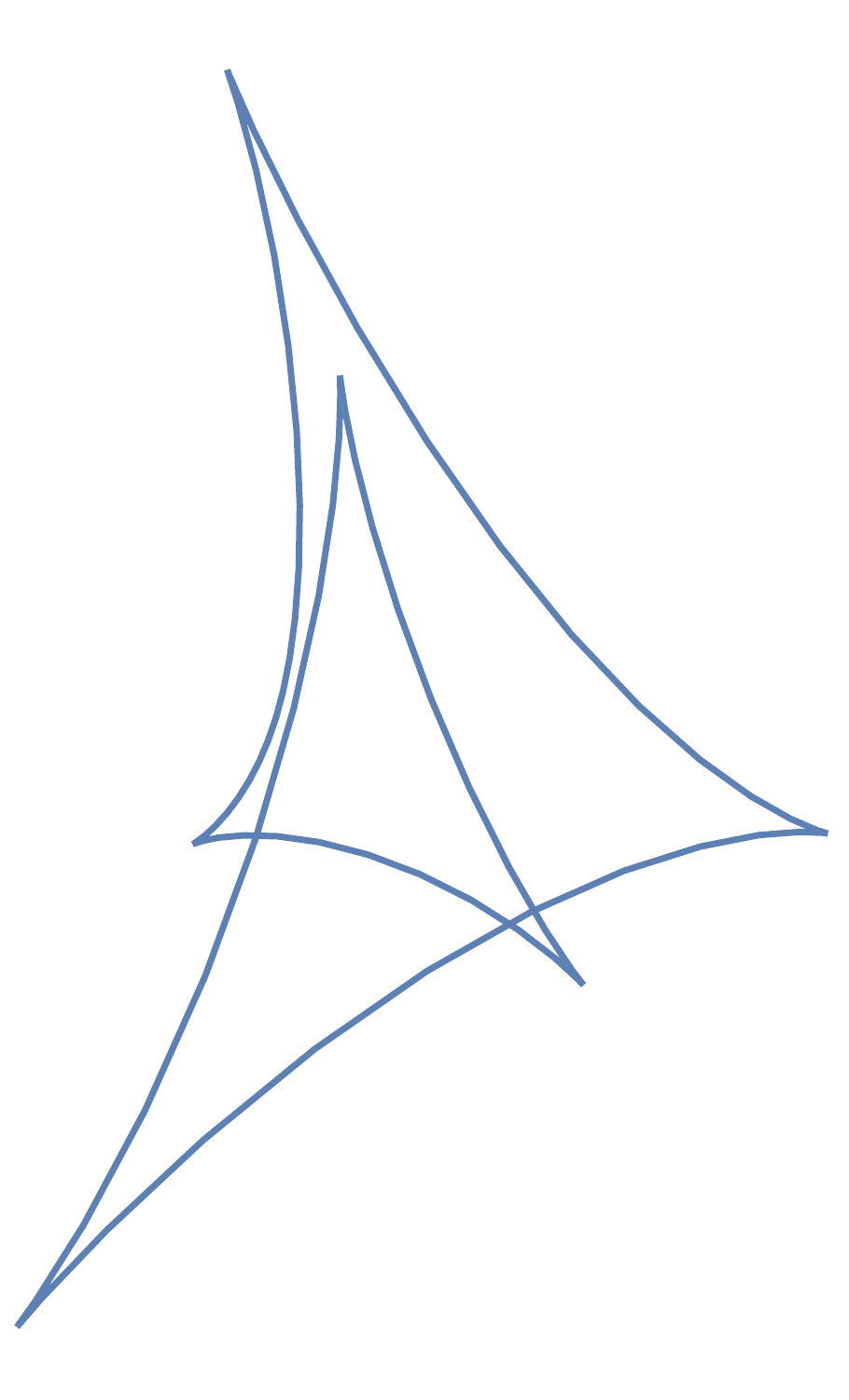}\includegraphics[width=0.25\textwidth]{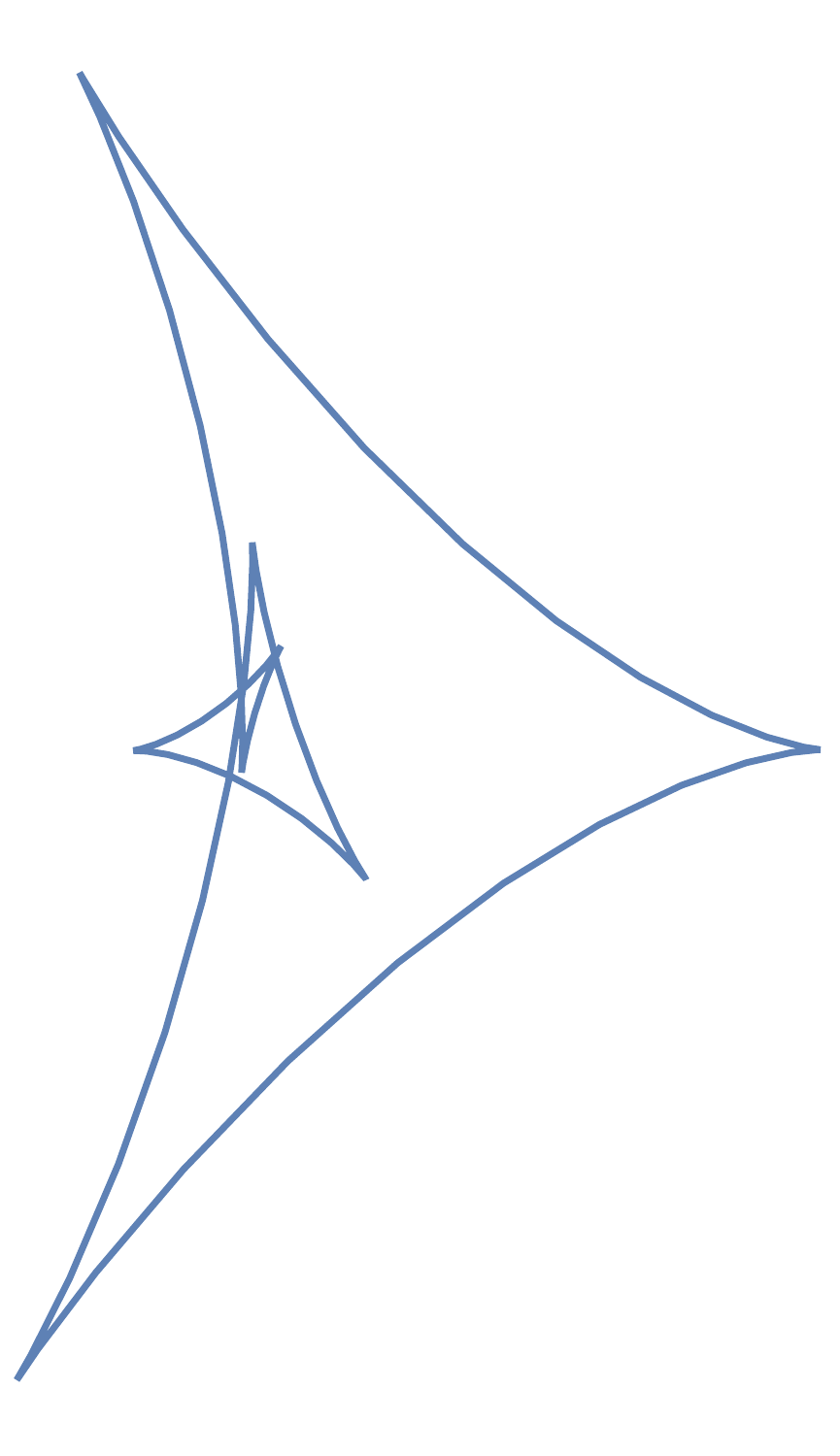}\includegraphics[width=0.28\textwidth]{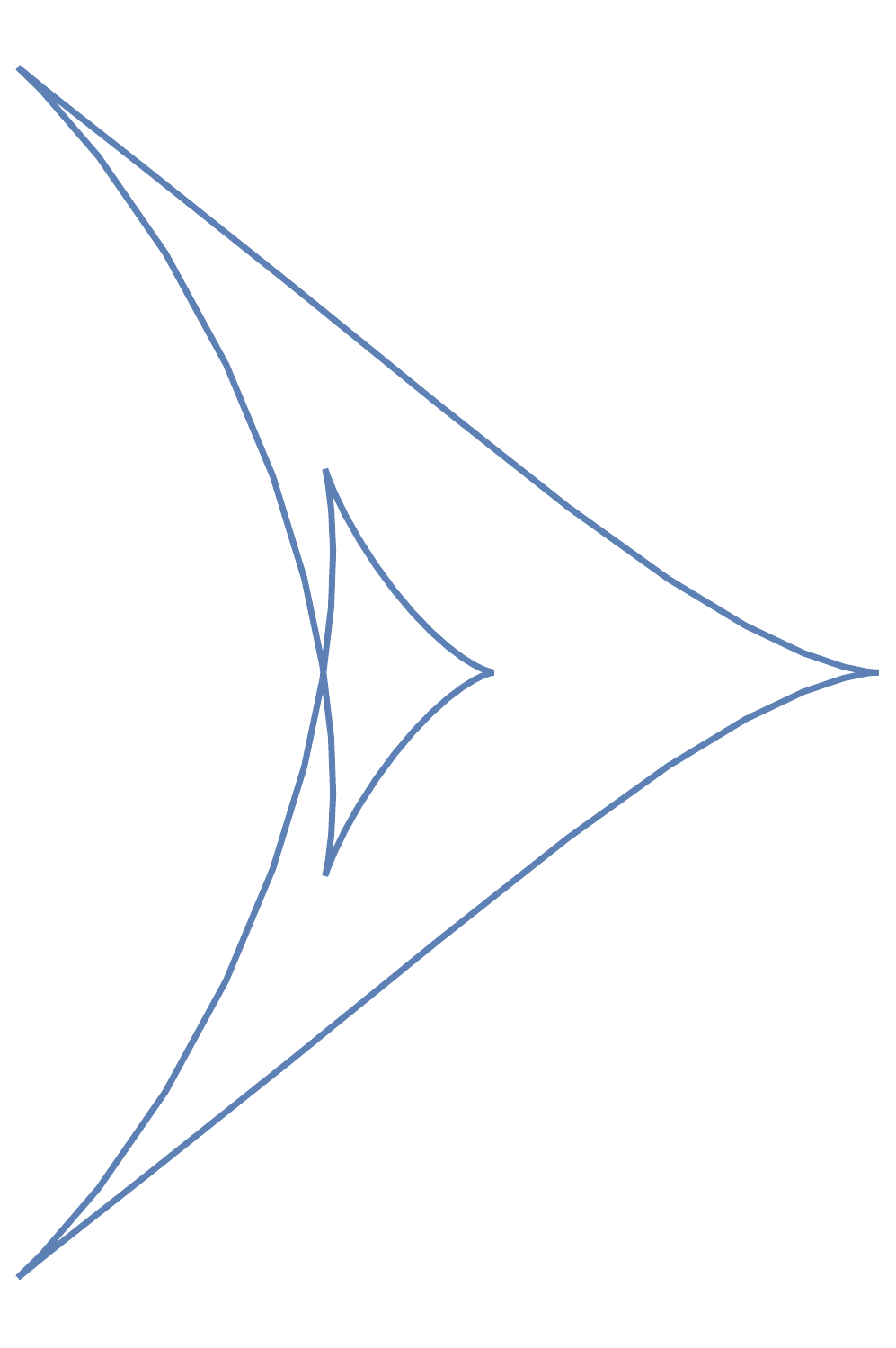}\includegraphics[width=0.235\textwidth]{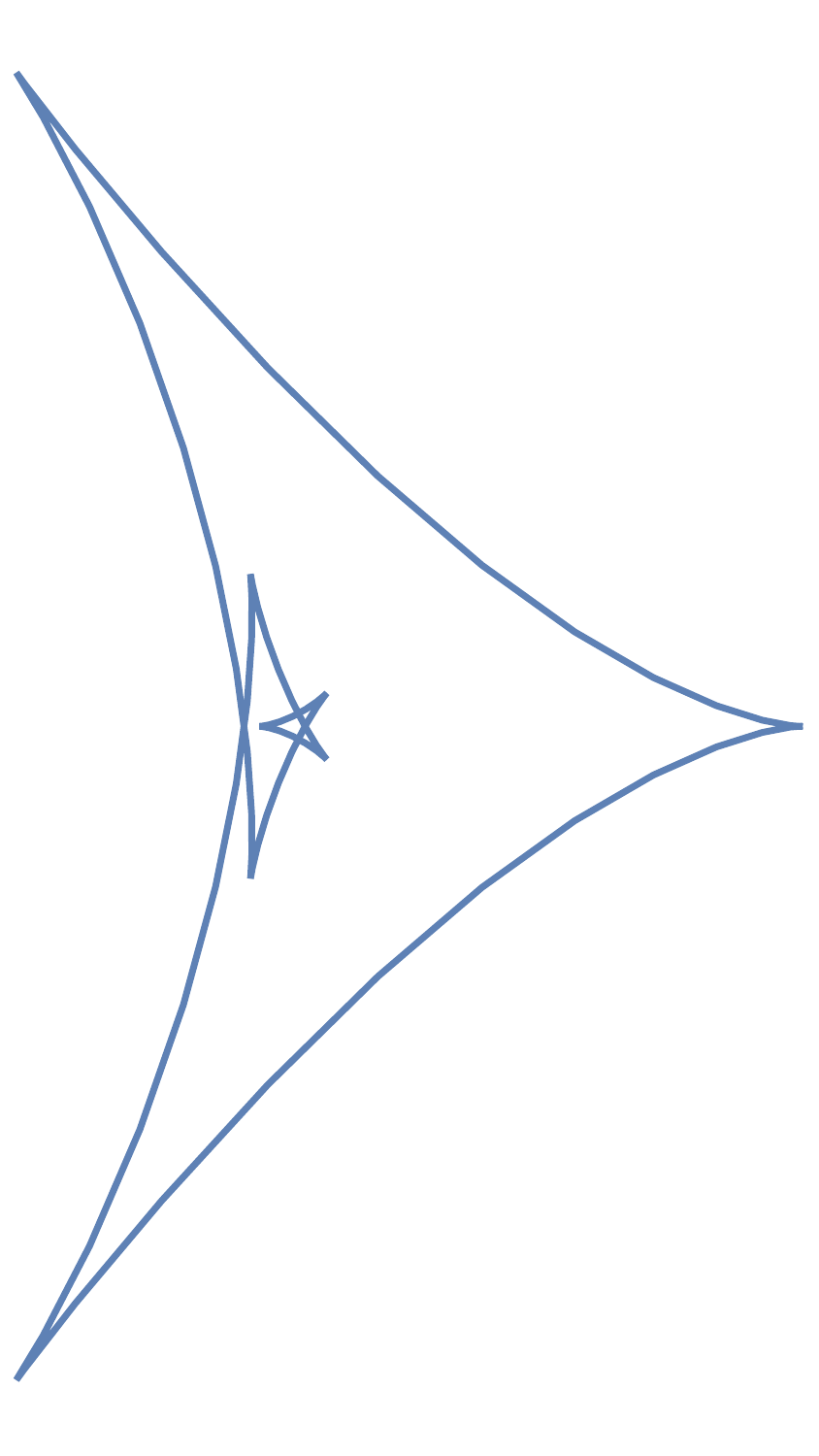}}}
\put(40,20){\hbox{(a)}}\put(140,20){\hbox{(b)}}\put(240,20){\hbox{(c)}}\put(340,20){\hbox{(d)}}
\end{picture}
\caption{The conjugate locus for various points on the $n=3$ sectoral harmonic surface (projected onto the tangent plane of the antipode); see Figure \ref{regions}.} \label{cpvarious}
\end{figure}
% conjugate locus sectoral 3
% (1.408,0.1) (1.536,0.1) (Pi/2,0.35) (Pi/2,0.16)

This paper will focus on the following phenomenon: let $p$ be a point in a smooth 2-dimensional surface $\mS$ and let $C_p$ be the conjugate locus of $p$ in $\mS$. As $p$ is moved on the surface the number of cusps of $C_p$ may vary (this is the `bifurcation' referred to in this paper's title). This is well known on the ellipsoid: as $p$ passes through an umbilic point, the conjugate locus degenerates from a curve with 4 cusps to simply a point (the antipodal umbilic). More elaborately, let us take the surfaces defined in terms of the spherical harmonic functions \cite{wangguo}, where for the sake of demonstration we will focus on the sectoral harmonics defined in polar form via (see \cite{TWspherical} for a proof that the geodesic flow on this surface is not integrable) \[ r(\theta,\phi)=1+\epsilon \sin^n\theta \cos(n \phi), \qquad (\theta,\phi)\in(0,\pi)\times(0,2\pi),\ \epsilon\in[0,1),\ n\in\mathbb{N}.  \] Now as $p$ is varied the conjugate locus may develop additional cusps where previously there was a smooth arc (for example (a) to (b) in Figure \ref{cpvarious}) or a cusp  ((c) to (d) in Figure \ref{cpvarious}). In fact the surface is divided into regions where the conjugate locus of a point in each region will have either 6 cusps or 8 cusps when $n=3$ (see Figure \ref{regions} in the Appendix); for larger values of $n$ the situation is more complex. 

In Section 2 we will first specify the two bifurcation scenarios of interest, and then we will show that conjugate points, cusps of the conjugate locus, and bifurcations of these cusps are determined by the vanishing of the first, second and third derivatives of the exponential map respectively (while we use the term `bifurcation' we will find the language of singularity theory more appropriate). Equations for these higher derivatives are derived in Sections 3 and 4. These equations are known in various forms in the General Relativity literature (\cite{bazanski},\cite{Hodg},\cite{AP},\cite{Vines}) where the Jacobi equation is known as the `geodesic deviation equation' \cite{wald}, however we go further and cast these equations in terms of scalar invariants which (i) makes the analysis clearer (ii) reduces each order to a single scalar ODE, and (iii) facilitates similar analyses on surfaces not defined in parameterised form. We go on in Section 5 to classify the cusps of the conjugate locus. We finish in Section 6 with some further comments.

We adopt the convention that indices are only used where necessary. We will focus on 2 dimensional smooth surfaces however equations \eqref{jacobi},\eqref{baz} and \eqref{j3eq} apply to manifolds of any dimension.

\section{The distance function and its singularities}

Let $\gamma(s,\psi)$ be a family of unit-speed geodesics emanating from $p\in\mS$, where $s\in\mathbb{R}$ parameterises each geodesic and $\psi\in(0,2\pi)$ labels the members of the family; as such $s,\psi$ parameterise the neighbourhood of $p$ and are called the geodesic polar coordinates. We shall let $x$ denote the exponential map defined via $\ex:T_p\mS\to\mS:\ex(s v_\psi)=\gamma(s,\psi)$ where $v_\psi=\dot{\gamma}(0,\psi)$ (here and throughout dots shall denote derivatives w.r.t.\ $s$). Consider the vector field $J_1=\partial x/\partial \psi$ restricted to some geodesic in the family $\psi=\psi_0$. $J_1$ is a Jacobi field and satisfies Jacobi's equation \begin{equation} \frac{D^2J_1^a}{\partial s^2}+R^a_{\ bcd}\dot{\gamma}^bJ_1^c\dot{\gamma}^d=0.\label{jacobi} \end{equation} If we let $\{T=\dot{\gamma}, N\}$ be an orthonormal frame parallel transported along radial geodesics, we can write $J_1=\eta_1 T+\xi_1 N$ and the Jacobi equation separates into two scalar ODE's: \beq \ddot{\eta}_1=0,\quad \ddot{\xi}_1+K\xi_1=0, \label{firstord} \eeq where $K$ in the Gauss curvature. To find the points conjugate to $p$ along $\psi=\psi_0$ we use initial data $\eta_1(0)=\dot{\eta}_1(0)=0$ and $\xi_1(0)=0,\ \dot{\xi}_1(0)=1$; hence the tangential component to $J_1$ is trivial and we need only focus on the normal component, $\xi_1$. If there is some $s=R$ such that  $\xi_1(R)=0$, then $x(R(\psi_0),\psi_0)$ is conjugate to $p$ along $\psi=\psi_0$. The set of points conjugate to $p$ for $\psi\in\mathbb{S}^1$ is the conjugate locus of $p$, denoted $C_p$.

The following is well known: there is a smooth curve in $T_p\mS$ (which may have more than one component, but we assume is not empty) parameterised by $\psi$ as $R=R(\psi)$, and if $R'(\psi_0)=0$ then the image of $R$ in $\mS$ (i.e.\ $C_p$) has a cusp at $x(R(\psi_0),\psi_0)$. As we vary the base point $p$, the curve $R=R(\psi)$ may develop or lose stationary points and hence $C_p$ may develop or lose cusps; this process is the `bifurcation' we are interested in.

\begin{wrapfigure}{r}{0.41\linewidth}
\vspace{-0.5cm}
  \includegraphics[width=0.4\textwidth]{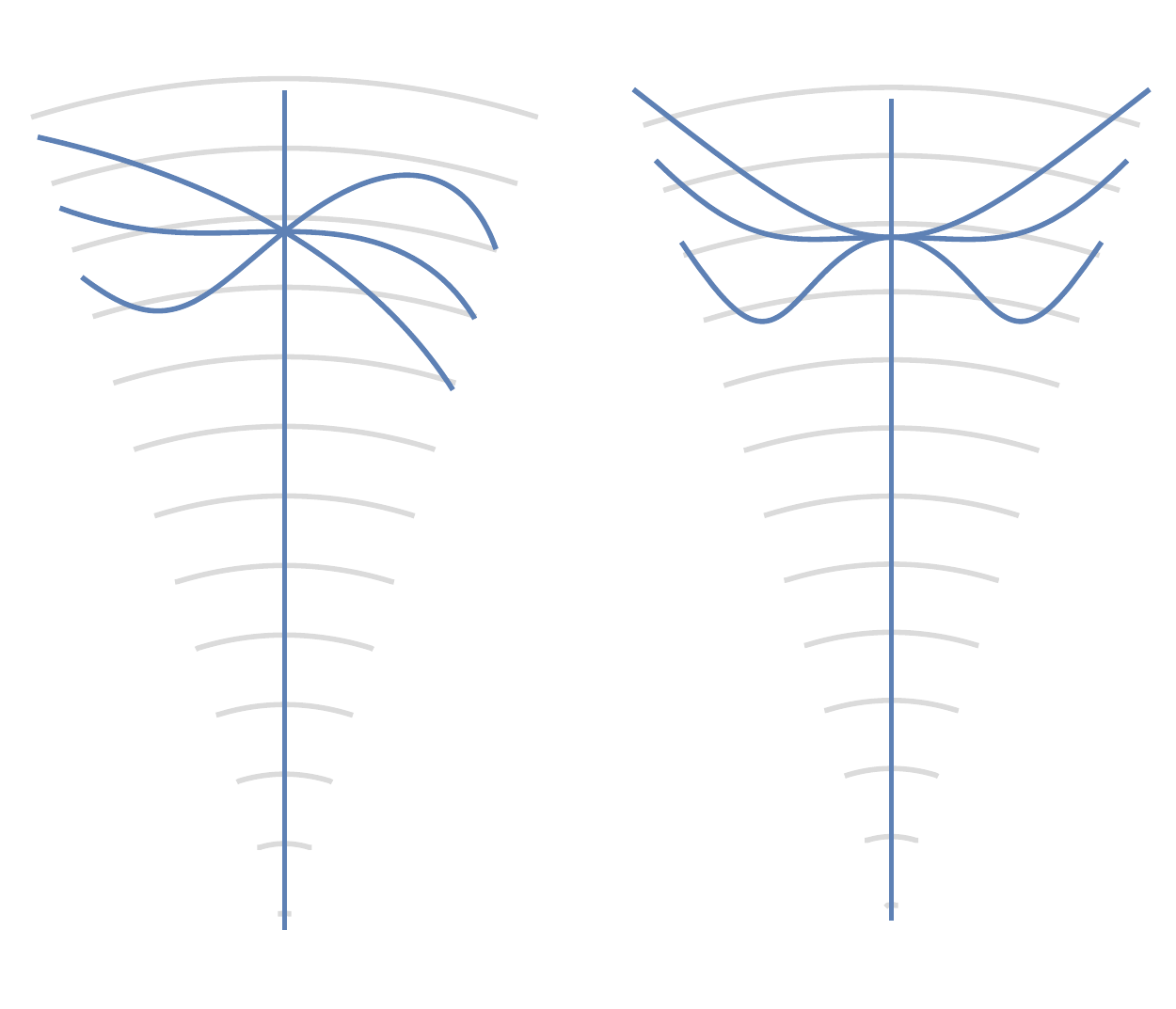} \caption{Representative sketches of the two bifurcation scenarios described in the text.}\label{bifpic}
  \vspace{-0.5cm}
\end{wrapfigure}
% plots

%\footnote{A function with an $A_k$ singularity can be embedded in a family of functions with parameters called an unfolding; the simplest form with the smallest number of parameters is a versal (or (p)versal or miniversal) unfolding \citep{bruce}}

Standard terminology is that $R$ has an $A_k$ singularity at $\psi_0$ if the first $k$ derivatives of $R$ vanish at $\psi_0$, but not the $(k+1)$th. We discard as uninteresting the $A_1$ singularity, as this simply represents a cusp moving along $C_p$. The versal unfolding of the $A_2$ singularity (i.e.\ the embedding of $R$ in a parameterized family of functions with the simplest form \cite{bruce}) is given by \[ \mu\, \delta\psi+R^{(3)}_0\frac{\delta\psi^3}{6}, \] and as the unfolding parameter $\mu$ passes through zero the number of stationary points of $R$ in the vicinity of $\psi_0$ passes from 0 to 2 (or vice versa), see the left of Figures \ref{bifpic}, \ref{cpvarious} and \ref{cusps}. We shall refer to this as the `arc' bifurcation, although the terms `fold' \cite{kuz} and `swallowtail' \cite{bruce} might also be appropriate. While the $A_3$ singularity may seem special it is also of interest, since when $\gamma(s,\psi_0)$ lies along a line of symmetry of $\mS$ then $R$ is an even function in $\delta\psi$, and the versal unfolding of the $A_3$ singularity is (assuming $p$ is moved along this line of symmetry) \beq \mu\frac{\delta\psi^2}{2}+R^{(4)}_0\frac{\delta\psi^4}{24} \label{symm}. \eeq As $\mu$ passes through zero the number of stationary points passes from 1 to 3 (and vice versa), see the right of Figures \ref{bifpic},\ref{cpvarious} and \ref{cusps}. We will refer to this as the `cusp' bifurcation, although the terms `pitchfork' and `butterfly' might also be appropriate.

The problem is we do not know what $R$ and its derivatives are, nor do we have any equations for them. We will now show that singularities of $R$ are due to the simultaneous vanishing of higher derivatives of the exponential map w.r.t.\ $\psi$, and in the next sections we will derive equations for these derivatives. 

\begin{prop} Suppose a geodesic $\gamma(s,\psi)$ emanating from $p$ reaches a conjugate point at $s=R(\psi)$, as described in the text. Then $R$ will have an $A_k$ singularity at $\psi$ if and  only if the first $k+1$ covariant derivatives of $\ex$ w.r.t.\ $\psi$ vanish at $\gamma(R(\psi),\psi)$. \end{prop}

%J_1 is a vector field over S, so J_1=J_1(s,\psi). There is some curve in S where J_1=0, this is the curve s=R(\psi) so J_1(R(\psi),\psi)=0. Now if we %differnetiate with respect to \psi, we are finding the derivative of J_1 as we move along this curve, but J_1 has the same value everywhere on it so its %derivative is zero. Also we are evaluating the derivative on the curve, so all derivatives are evaluated on s=R(\psi). 
%Example: vf = {(x^2 + y^2/2 - 1), (x^2 + y^2/2 - 1)}; b = 1.5; 
%Show[VectorPlot[vf, {x, -b, b}, {y, -b, b}],ContourPlot[Sqrt[vf[[1]]^2 + vf[[2]]^2] == 0.1, {x, -b, b}, {y, -b, b}]]

\begin{proof} By definition, $R=R(\psi)$ is the value of $s$ along $\gamma(s,\psi)$ where $\partial x/\partial\psi\equiv J_1$ vanishes, i.e. \[ J_1^a(R(\psi),\psi)=0. \] Taking the covariant derivative w.r.t. $\psi$, \[ \frac{DJ^a_1}{\partial s}R'+\frac{DJ_1^a}{\partial \psi}=0, \] where all derivatives are evaluated on $s=R(\psi)$. Hence $R$ has an $A_1$ singularity iff the second derivative of the exponential map w.r.t.\ $\psi$ vanishes (since $DJ_1^a/\partial s\neq 0$ on $s=R$). Differentiating again, \[ \frac{DJ_1^a}{\partial s}R''+R'\left(\frac{D^2J_1^a}{\partial s^2}R'+\frac{D^2J_1^a}{\partial \psi \partial  s}+\frac{D^2J_1^a}{\partial s \partial  \psi} \right)+\frac{D^2J_1^a}{\partial \psi^2}=0. \] Hence $R$ has an $A_2$ singularity iff the second and third derivatives of $\ex$ vanish; continuing this way the proposition follows by induction. \end{proof}

The following corollary is immediate:

\begin{corr} Along a geodesic emanating from $p$: there is a conjugate point where the first derivative of $\ex$ w.r.t.\ $\psi$ vanishes; there is a cusp of the conjugate locus where the first and second derivatives vanish; there is an arc bifurcation where the first, second and third derivatives vanish; and if the geodesic lies in a symmetry plane of $\mS$ (and $p$ is moved along this geodesic), there is a cusp bifurcation where the first, second and third derivatives vanish.
\end{corr}

\section{The second derivative of the exponential map}

The advantage of phrasing bifurcations in terms of the vanishing of derivatives of $\ex$ is that we can now derive equations for these derivatives. These equations are in tensorial form and as such can be expressed in any coordinate system, however for generic coordinate systems this would be very cumbersome. Instead we can write the equations in terms of scalar invariants, in analogy with \eqref{firstord}, which allows for easy computation; in fact we will show the analysis reduces to a single scalar equation.

Let $J_2=D J_1/\partial \psi$, then  \[ \frac{D}{\partial \psi}\frac{D}{\partial s}\left(\frac{D J_1^a}{\partial s}\right)-\frac{D}{\partial s}\frac{D}{\partial \psi}\left(\frac{D J_1^a}{\partial s}\right)=R^a_{\ bcd}\left(\frac{D J_1^b}{\partial s}\right)J_1^c \dot{\gamma}^d, \] and, using the Jacobi equation,  \[ \frac{D}{\partial \psi}\left(-R^a_{\ bcd}\dot{\gamma}^b J_1^c \dot{\gamma}^d \right)-\frac{D}{\partial s}\left(\frac{DJ_2^a}{\partial s}+R^a_{\ bcd}J_1^b J_1^c \dot{\gamma}^d \right)=R^a_{\ bcd}\left(\frac{D J_1^b}{\partial s}\right)J_1^c \dot{\gamma}^d. \] Using the symmetries of the Riemann tensor this reduces to  \begin{equation} \frac{D^2J_2^a}{\partial s^2}+R^a_{\ bcd}\dot{\gamma}^b J_2^a \dot{\gamma}^d=(R^a_{\ bcd;e}+R^a_{\ ecd;b})\gd^b \gd^c J_1^d J_1^e+4R^a_{\ bcd}\frac{DJ_1^b}{\partial s}\gd^c J_1^d.\label{baz} \end{equation}  This is known as `Ba\.{z}a\'{n}ski's equation' and appears in \cite{bazanski}. Decomposing $J_2$ into a tangential and normal component via $J_2=\eta_2T+\xi_2N$, we express this equation in geodesic polar coordinates and then identify the coefficients in terms of scalar invariants. We find the two following equations: \beq \ddot{\eta}_2=4K\xi_1\dot{\xi}_1+(T^a\partial_a K)\xi_1^2,\quad \ddot{\xi}_2+K\xi_2=-(N^a\partial_a K)\xi_1^2.\label{secord} \eeq Based on the previous section, we know that to identify a cusp of $C_p$ we require both $\eta_2$ and $\xi_2$ to vanish simultaneously on $s=R(\psi)$. Looking at the tangential component first, we see that $\eta_2$ cannot be trivial (unlike at first order), however we can still find an exact solution: it is \[ \eta_2=-\xi_1\dot{\xi}_1. \] Since $\xi_1$ vanishes on $C_p$ this means we need only look to the vanishing of the normal component, $\xi_2$, to identify a cusp of $C_p$.  Some immediate (well-known) observations follow: the conjugate locus of a pole on a surface of revolution is a point (or empty), and a conjugate point along a line of symmetry must be a cusp of the conjugate locus. Both observations follow from \eqref{secord} since in each case $N^a\partial_a K=0$.

Now we can understand the creation or annihilation of cusps in the following intuitive picture: for each $p$ we consider the $\xi_1=0$ and $\xi_2=0$ contours in $T_p\mS$; the intersections between these curves mark the cusps of the conjugate locus. As $p$ moves in the surface these two contours will vary leading to bifurcations. An `arc' bifurcation is shown in Figure \ref{contours}; we see the two contours come to intersect one another transversally leading to the creation of two new cusps of the conjugate locus.

\begin{figure}
  \centering
 {\includegraphics[width=0.55\textwidth]{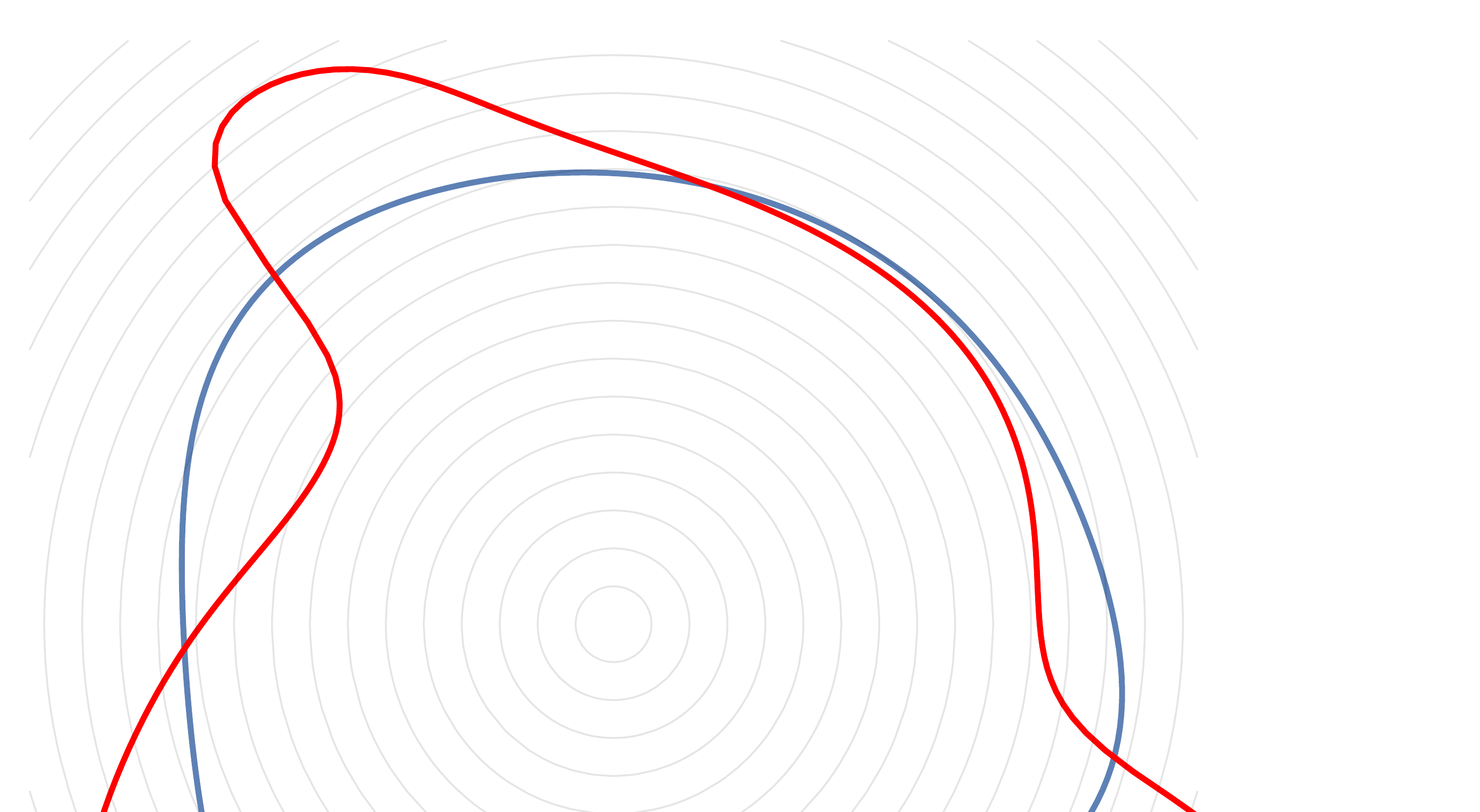}\hspace{-1cm}\includegraphics[width=0.5\textwidth]{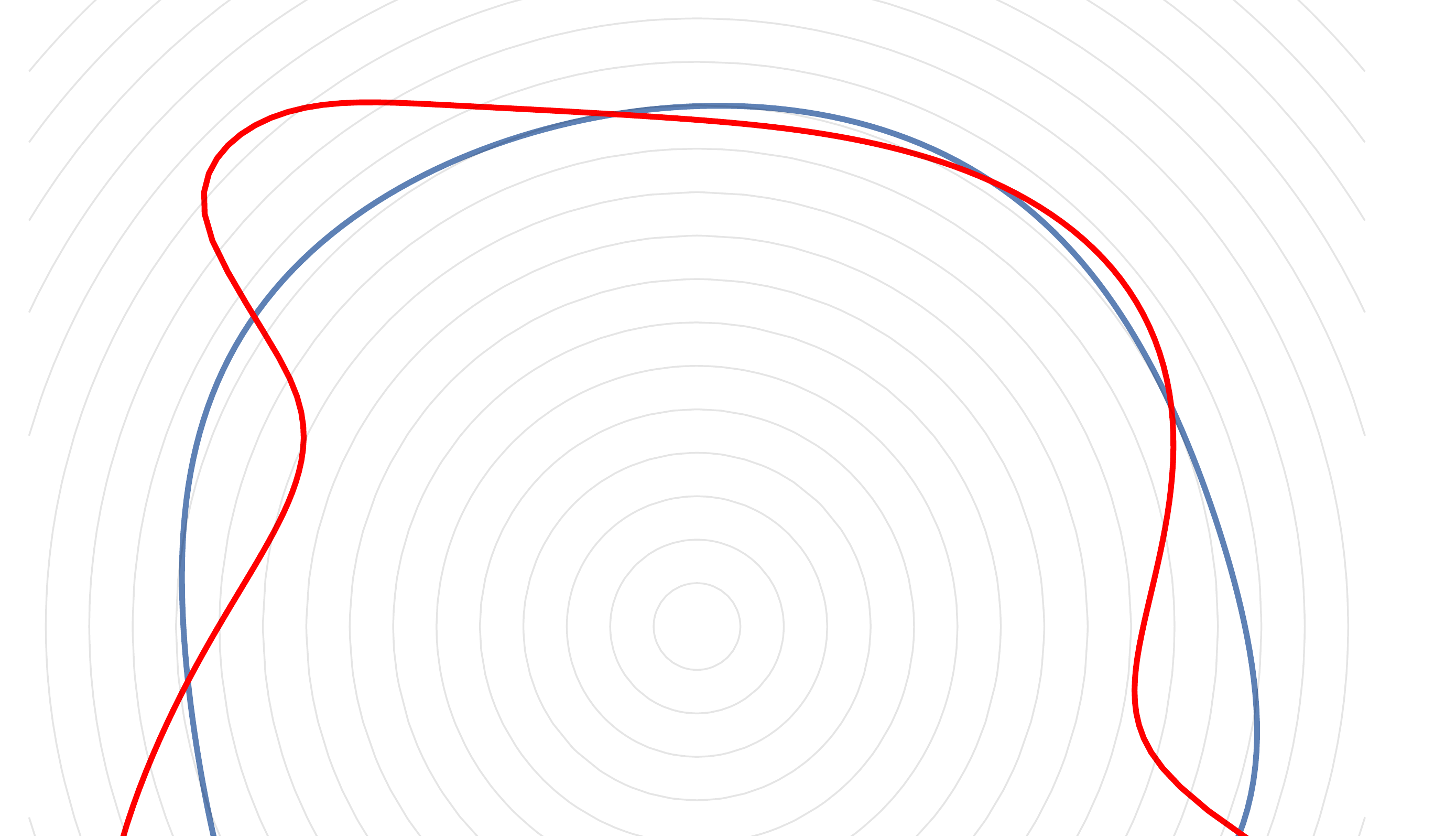}} \caption{The $\xi_1=0$ (blue) and $\xi_2=0$ (red) contours in $T_p\mS$ for the two points labeled (a) and (b) in Figure \ref{regions}. Note: two more intersections are not shown.} \label{contours}
\end{figure}
% second order bazanski sectoral 3 plots for paper

\section{The third derivative of the exponential map}

While the picture from the previous section is informative, we still need to view a sector of geodesics to identify a bifurcation, rather than simply follow an individual geodesic. However in Section 2 we have shown that there is an $A_2$ singularity in $R$ at $\psi=\psi_0$ if the third derivative of the exponential map also vanishes at $x(R(\psi_0),\psi_0)$; we show in this section that this again reduces to a single scalar equation.

Let $J_3=DJ_2/\partial \psi$, and following the procedure of the previous section we find an equation for $J_3$ which is in the Appendix \eqref{j3eq}. As before we write $J_3=\eta_3T+\xi_3N$ and equation \eqref{j3eq} separates into two scalar equations: \beq
  \ddot{\eta}_3=6K(\xi_1\xi_2)_{,s}+3(T^a\partial_a K)\xi_1\xi_2+6(N^a\partial_a K)\xi_1^2\dot{\xi}_1+(T^aN^b\nabla_a\partial_b K)\xi_1^3
\eeq and \begin{align}
  \ddot{\xi}_3+K\xi_3=&-(N^aN^b\nabla_a\partial_b K)\xi_1^3-2K^2\xi_1^3-3(N^a\partial_a K)\xi_1\xi_2 \nonumber  \\  &+3(T^a\partial_a K)\xi_1^2\dot{\xi}_1+6K\xi_1\dot{\xi}_1^2. \label{xitilde}
\end{align} If we wish to identify when an $A_2$ singularity of $R$ occurs along a particular geodesic we would need both $\eta_3$ and $\xi_3$ to vanish simultaneously at a cusp ($\xi_1=\xi_2=0$). We can again find an exact solution for the tangential component, \[ \eta_3=-\xi_1\dot{\xi}_2-2\dot{\xi}_1\xi_2, \] which vanishes at cusps of $C_p$, so we need only look for the vanishing of the normal component, $\xi_3$.

Consider the following experiment: on the triaxial ellipsoid, we allow the point $p$ to move along the ``middle'' ellipse, passing through an umbilic point. We consider the geodesic which emanates from $p$ along the line of symmetry; as such the $\xi_1\xi_2$ term vanishes in \eqref{xitilde}, and only the first order terms contribute to the third order equation. We simultaneously solve the geodesic equations, the $\xi_1$ equation \eqref{firstord}, and the $\xi_3$ equation \eqref{xitilde} until the value of $s=R$ where $\xi_1(R)=0$, then record $\xi_3(R)$ for each $p$. As expected, $\xi_3(R)$ passes through zero as $p$ passes through the umbilic point and as such there occurs a cusp bifurcation with the annihilation of 2 cusps. We can say the same for the symmetric geodesic traveling in the opposite direction, and it would not take much more to show that the conjugate locus degenerates to a point. While we won't go further as this is a well known phenomenon on the ellipsoid, we emphasise that the methods of this paper do not rely on the existence of an integral of the geodesic equations.

\begin{wrapfigure}{r}{0.41\linewidth}
\vspace{-0.5cm}
  \includegraphics[width=0.4\textwidth]{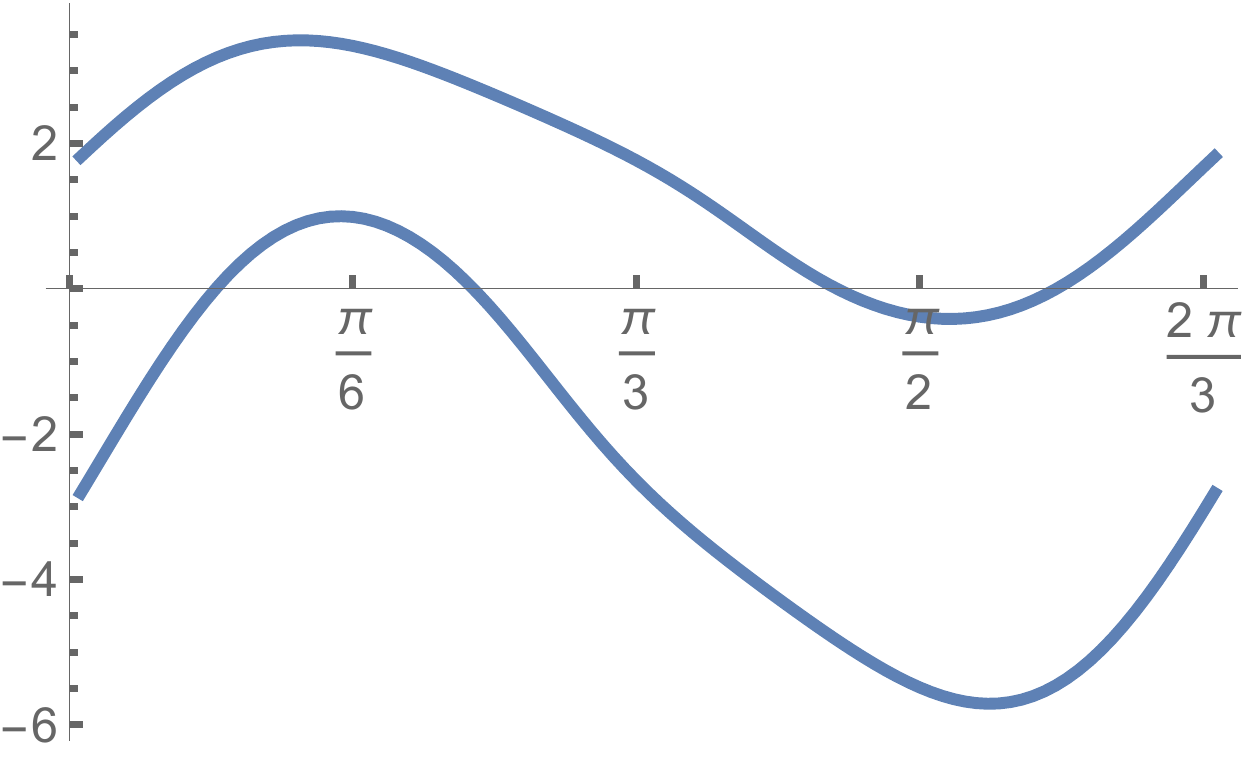} \caption{$\phi_0$ versus $\xi_3(R)$ for the equatorial geodesic as described in the text.}\label{xibb}
  \vspace{-0.5cm}
\end{wrapfigure}
% 3rd order sectoral 3 equatorial

As a more elaborate example, let us consider a surface for which we know the geodesic flow is not integrable: the $n=3$ spherical harmonic surface described in the Introduction and shown in Figure \ref{regions}. We focus on the symmetric equatorial geodesic, i.e.\ $p=(\pi/2,\phi_0)$ with $\dot{\theta}(p)=0$, and allow $\phi_0$ to vary. For each value of $\phi_0$ we record $\xi_3(R)$ as described in the previous paragraph, and we show the results in Figure \ref{xibb}. Notice there are two curves, one for each of $\dot{\phi}(0)>0$ and $<0$. The vanishing of $\xi_3(R)$ for certain values of $\phi_0$ shows the locations of the cusp bifurcations, in agreement with Figure \ref{regions}. Now we can intuitively understand the mechanism for this bifurcation: for certain base points $p$ the third derivative of the exponential map vanishes before the conjugate point, for some $p$ after the conjugate point. At a certain $p$ therefore, the third derivative vanishes {\it at} the conjugate point and this high order focusing leads to singularity in $R$ and generates a change in the number of cusps of $C_p$.

\section{Classification of the cusps}

Based on the previous analysis we can now classify the cusps of the conjugate locus according to the singularity of $R$. We will denote by $\alpha(\psi)=x(R(\psi),\psi)$ the image of $R$ in $\mS$, i.e.\ the conjugate locus of $p$. Let $q$ be a point on this curve, and the following series \[ \left.\frac{d\alpha^a}{d\psi}\right|_q\delta\psi+\left.\frac{D}{d\psi}\frac{d\alpha^a}{d\psi}\right|_q\frac{\delta\psi^2}{2}+\ldots \] is the projection into $T_q\mS$ of  the Taylor series of $C_p$ at $q$. Now if we consider the $T,N$ vectors parallel propagated along the radial geodesic $\gamma(s,\psi)$, then $T(q),N(q)$ form a basis for $T_q\mS$, and hence the series just given can be written as a linear combination of these orthonormal vectors; from the leading terms in this series we can discern the local structure of $C_p$ at $q$. We shall use the notation $(n,m)$ to mean the series representation of $\alpha$ has leading terms of order $\delta\psi^n$ and $\delta\psi^m$ in the directions $T$ and $N$ respectively.

\begin{prop}
  If $R(\psi)$ has an $A_k$ singularity at $\psi=\psi_0$ then the conjugate locus is $(k+1,k+2)$ at $\alpha(\psi_0)$.
\end{prop}

\begin{proof}
  Since $\alpha(\psi)=x(R(\psi),\psi)$, the tangent vector to $C_p$ is \[ \frac{d\alpha^a}{d\psi}=\frac{Dx^a}{\dd s}R'+\frac{Dx^a}{\dd \psi}=R'T^a+J_1^a\equiv R'T^a. \] Taking the covariant derivative we see \[ \frac{D}{d\psi}\frac{d\alpha^a}{d\psi}=R''T^a+R'\frac{DT^a}{\dd\psi}=R''T^a+R'\frac{DJ_1^a}{\dd s}. \]  Continuing in this manner it is easy to derive the following formula: \[ \frac{D^n}{d\psi^n}\frac{d\alpha^a}{d\psi}=R^{(n+1)}T^a+nR^{(n)}\frac{DJ_1^a}{\dd s}+\sum_{k=1}^{n-1}R^{(k)}G^a_k \quad n\geq 2 \] where $G^a_k$ is a progressively more and more complicated expression which we omit here (but use for Figure \ref{cusps}). Now since $DJ_1^a/\dd s=\dot{\xi}_{1}N^a$ we can write the leading terms of the series expansion of $C_p$ at $q$ when $R$ has an $A_k$ singularity as \[ T^a\left[R^{(k+1)}\frac{\delta\psi^{k+1}}{(k+1)!}+\ldots \right]+N^a\left[(k+1)R^{(k+1)}\dot{\xi}_{1}(R)\frac{\delta\psi^{k+2}}{(k+2)!}+\ldots \right] \]
\end{proof}

Hence if $R'\neq 0$ (an $A_0$ singularity of $R$) then locally $C_p$ is a parabola opening in the direction of $N$ if $R'<0$ and $-N$ if $R'>0$ (note $\dot{\xi}_{1}(R)<0$), and at an $A_1$ singularity $C_p$ has an `ordinary' cusp, which points towards or away from $p$ if $R''$ is positive or negative respectively, as expected \cite{myers}. If we recall the arc bifurcation $A_0\to A_2\to A_0$,  we see that as $R'$ passes through zero the conjugate locus goes through $(1,2)\to(3,4)\to(1,2)$. On the other hand, at a cusp bifurcation $A_1\to A_3\to A_1$, we see that the conjugate locus passes through $(2,3)\to(4,5)\to(2,3)$ as $R''$ passes through zero. Representative sketches are given in Figure \ref{cusps}.

\begin{figure}
  \centering
  \includegraphics[width=\textwidth]{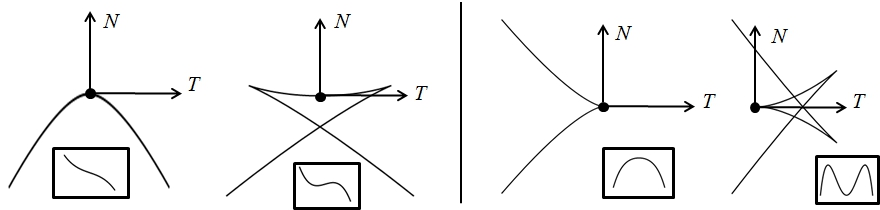} \caption{The local picture of $C_p$ for the arc (left) and cusp (right) bifurcations. A sketch of $R$ is given in the boxes.} \label{cusps}
\end{figure}
% pics for paper & plots

\section{Conclusions}

The main finding of this paper is we have shown that we can understand the creation and annihilation of cusps of the conjugate locus as due to higher order focusing of neighbouring geodesics along a particular geodesic. We have derived relatively simple equations for the second and third derivatives of the exponential map and have shown how the normal component contains all the relevant information. We have demonstrated this theory on the ellipsoid and spherical harmonic surfaces. Furthermore we have classified the cusps of the conjugate locus paying particular to the arc and cusp bifurcation.

\begin{wrapfigure}{r}{0.41\linewidth}
\vspace{-0.5cm}
  \includegraphics[width=0.4\textwidth]{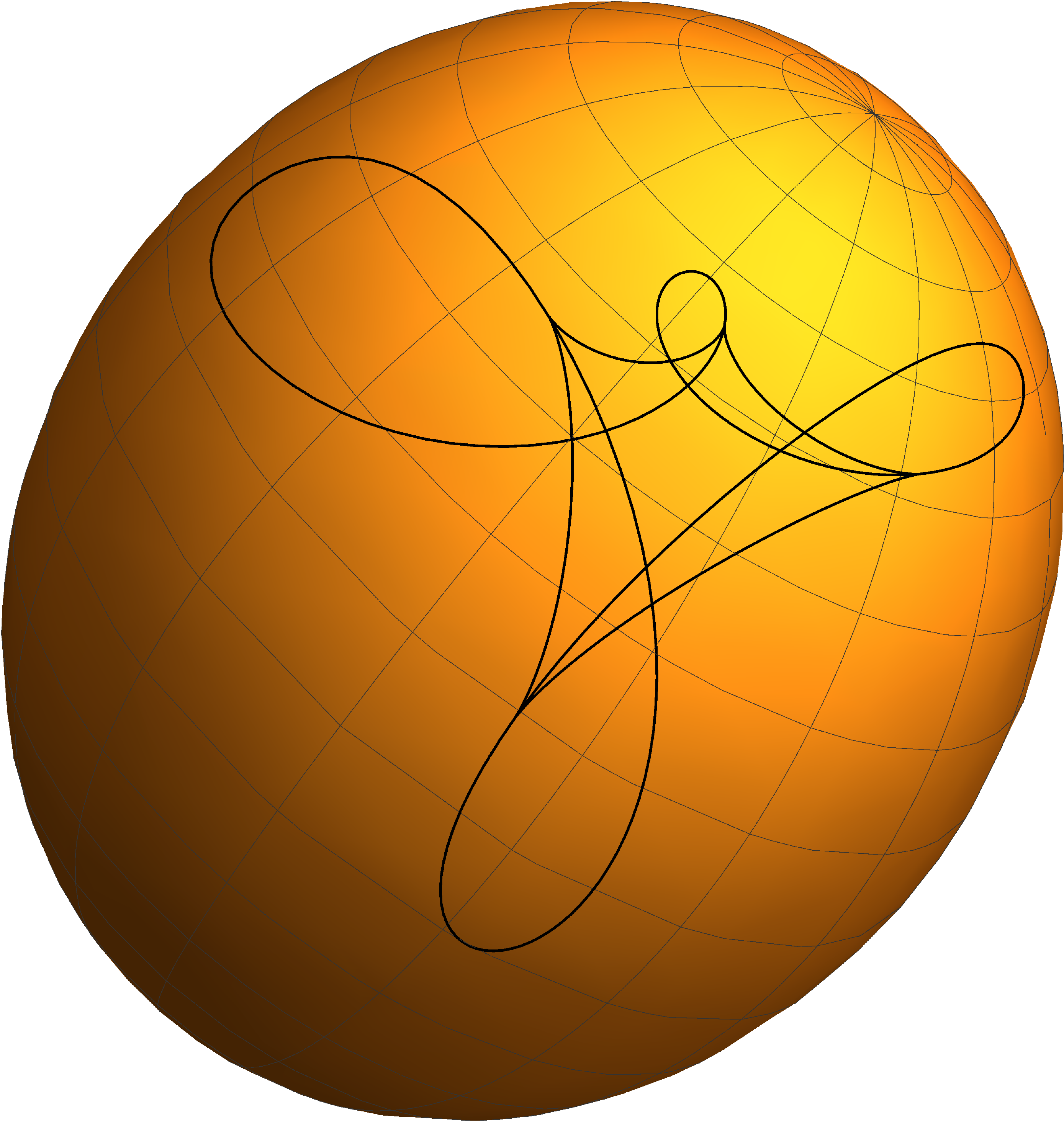} \caption{For a typical point on the triaxial ellipsoid, the image of the $\xi_1=0$ contour (i.e.\ the conjugate locus) and the $\xi_2=0$ contour.}\label{loops}
  \vspace{-0.5cm}
\end{wrapfigure}

% 2nd order bazanski ellipsoid

The methods of this paper can be carried forward to answer other questions. For example, in Figure \ref{contours} we see the $\xi_1=0$ and $\xi_2=0$ contours in $T_p\mS$. We know the image of the $\xi_1=0$ contour in $\mS$ has cusps, what of the image of the $\xi_2=0$ contour; does this have cusps? The answer in general is no: this curve is made up of a number of loops, as many as there are cusps to $C_p$, and each loop passes through a cusp of $C_p$ (see Figure \ref{loops}). However cusps do develop at precisely the moment of a bifurcation in the conjugate locus. To see this, let $\rho=\rho(\psi)$ be the value of $s$ where $\xi_2=0$ for each $\psi$; the image of this curve in $\mS$ is $\beta(\psi)=x(\rho(\psi),\psi)$ whose tangent vector is  \[ \frac{d\beta^a}{d\psi}=\frac{\dd x^a}{\dd s}\rho'+\frac{\dd x^a}{\dd \psi}=T^a\rho'+J_1^a \] and since $J_1$ does not in general vanish on $\rho(\psi)$, $\beta^a$ will be regular unless both $\rho'$ and $J_1$ vanish simultaneously; but this is precisely what happens at bifurcation (since the $R$ and $\rho$ curves touch in $T_p\mS$).

While not studied in this paper, a deeper understanding of the conjugate locus $C_p$ can lead to insight regarding the cut locus $K_p$. For example, it is well known that endpoints of the cut locus are cusps of the conjugate locus \cite{poincare}\cite{myers}, and it might be assumed that bifurcations in $C_p$ where extra pairs of cusps are generated must lead to the development of additional branches of $K_p$. Indeed it is tempting to assume that the cut locus has half as many endpoints as the conjugate locus has cusps. This is not necessarily true however: close scrutiny of some of the bifurcations described in this paper clearly show extra cusps to $C_p$ may develop without extra branches of $K_p$ (this is based on the observation that if we consider the family of geodesic circles centred on $p$, the cusps of these circles trace out $C_p$ whereas the self-intersections trace out $K_p$). It seems we can only accept that if $C_p$ has $2n$ (ordinary) cusps then $K_p$ has $\leq n$ endpoints. It is perhaps worth further investigation of the bifurcations in the cut locus.

In conclusion we have developed techniques to gain a deeper understanding of the conjugate locus on smooth surfaces without relying on the existence of additional integrals of the geodesic flow.

\appendix

%\begin{appendices}

\section{$J_3$ equation}

 \begin{align}
\frac{D^2J_3^a}{\dd s^2}+R^a_{\ bcd}\gd^bJ_3^c\gd^d=&(-3R^a_{\ bcd;e}-3R^a_{\ ced;b})\gd^bJ_2^c\gd^dJ_1^e+6R^a_{\ bcd}\left(\frac{DJ_1^b}{\dd s}\gd^cJ_2^d+\frac{DJ_2^b}{\dd s}\gd^cJ_1^d\right) \nonumber  \\ &+\left[ \left(R^a_{\ bcd;e}+R^a_{\ ecd;b}\right)_{;f}+4R^a_{\ gcd}R^g_{\ efb}  \right]\gd^b\gd^cJ_1^dJ_1^eJ_1^f \nonumber \\ & +\left(6R^a_{\ bcd;e}+2R^a_{\ ebd;c} \right)\frac{DJ_1^b}{\dd s}\gd^cJ_1^dJ_1^e +4R^a_{\ bcd}\frac{DJ_1^b}{\dd s}\frac{DJ_1^c}{\dd s}J_1^d. \label{j3eq}
\end{align}

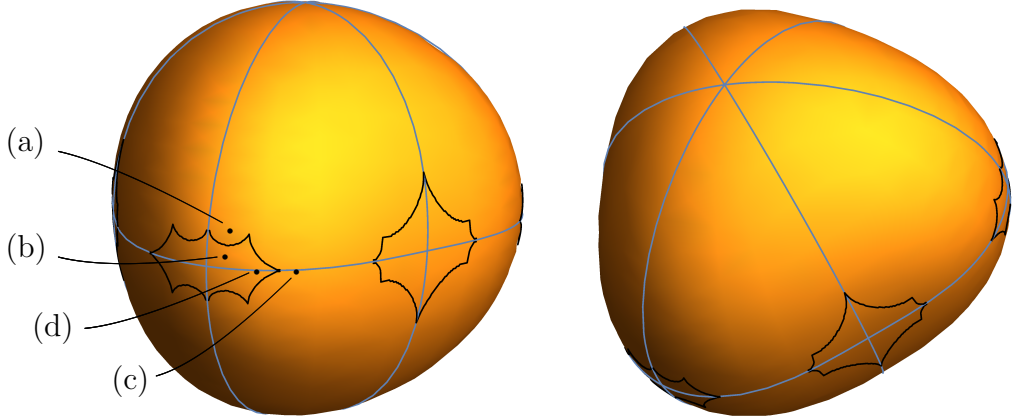
\begin{figure}[h]
\begin{picture}(200,170)
\put(50,0){{\includegraphics[width=0.4\textwidth]{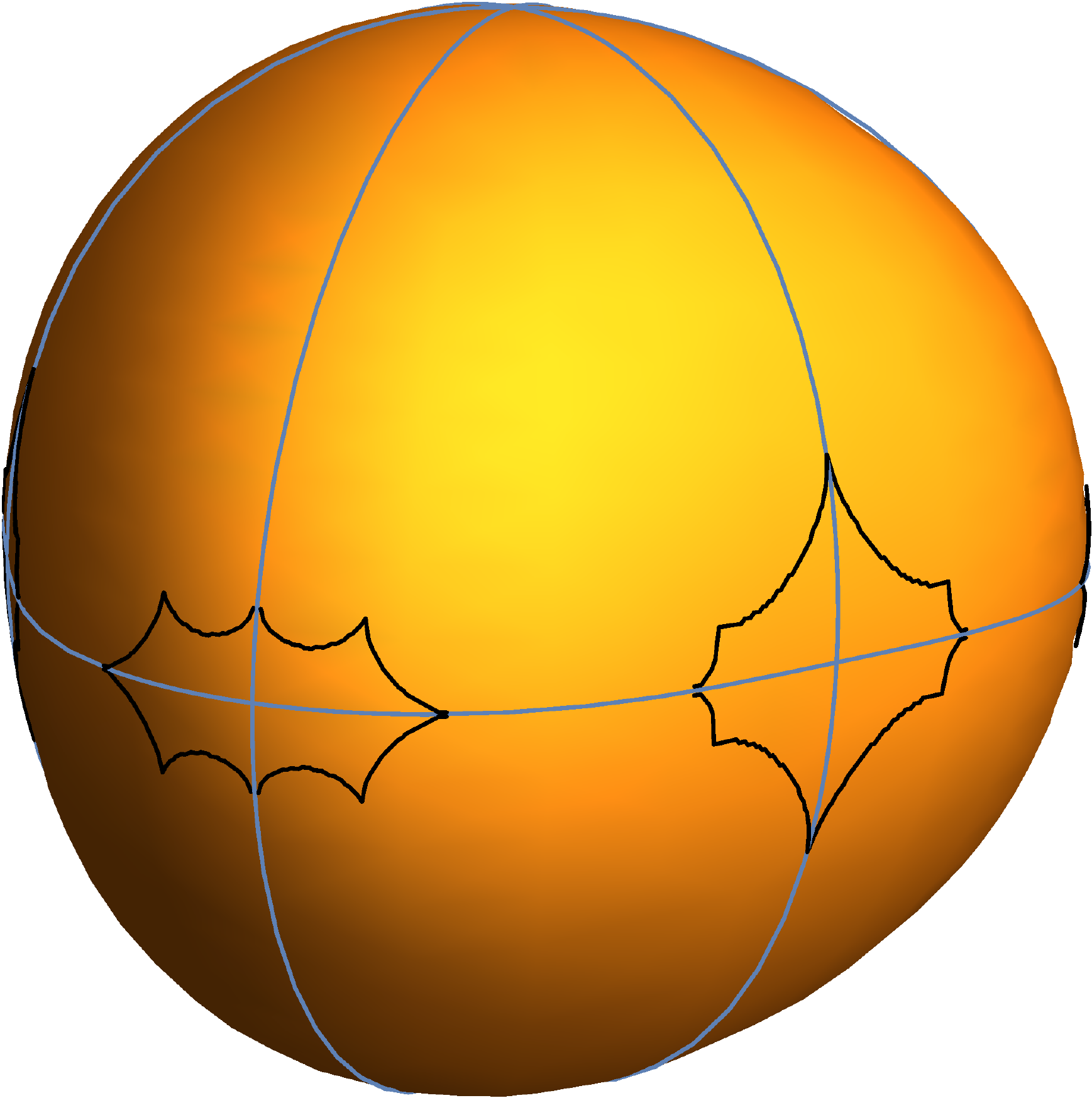}\hspace{1cm}\includegraphics[width=0.4\textwidth]{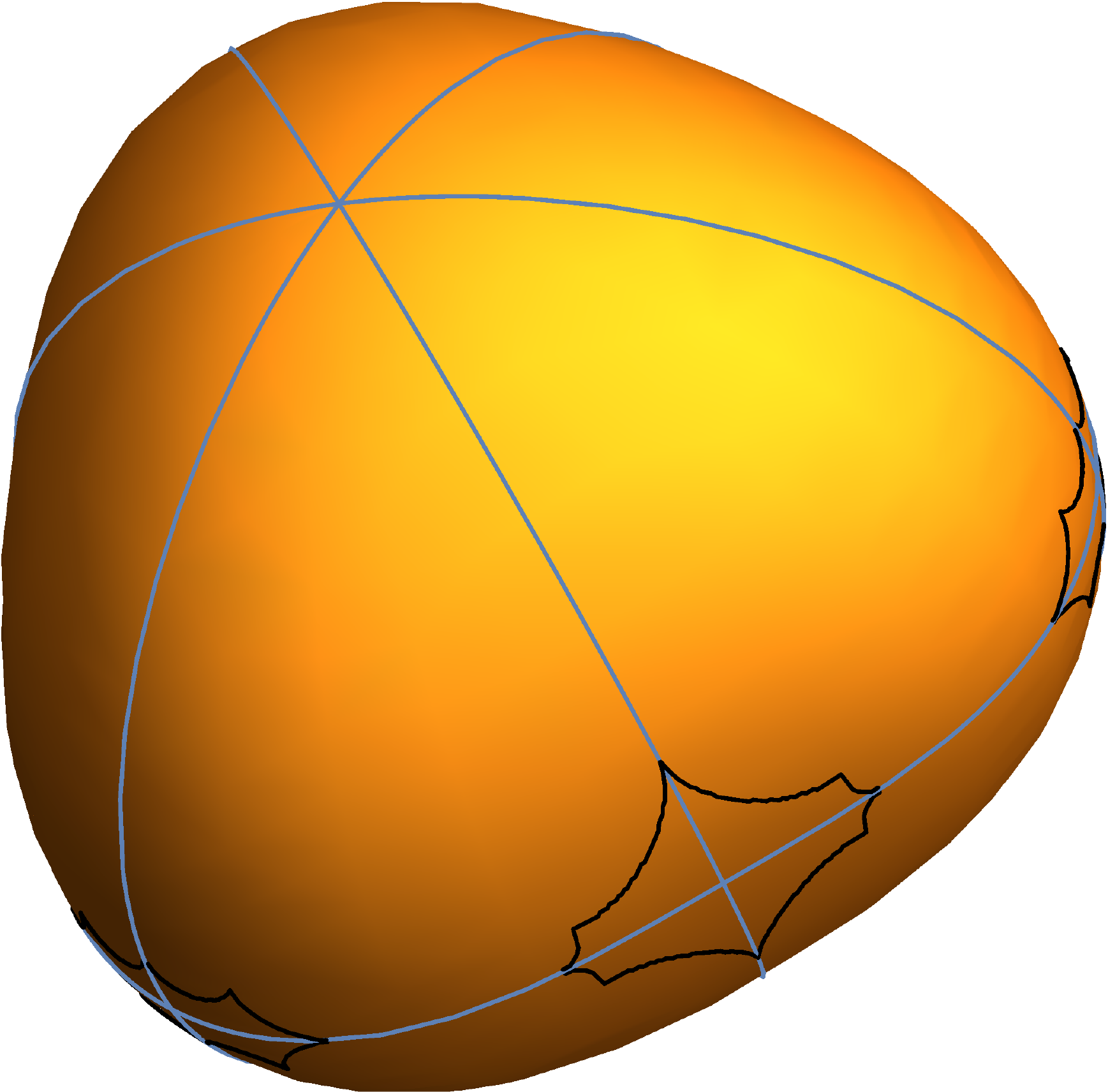}}}
\put(10,100){\hbox{(a)}}\put(10,60){\hbox{(b)}}\put(50,10){\hbox{(c)}}\put(20,30){\hbox{(d)}}
\put(95,70){\circle*{2}}\put(93,60){\circle*{2}}\put(120,54.5){\circle*{2}}\put(105,54.5){\circle*{2}}
\qbezier(30,100)(52,95)(92,73)\qbezier(30,60)(50,55)(90,60)\qbezier(40,33)(60,35)(102,54)\qbezier(67,15)(85,20)(118,52)
\end{picture}
\caption{The $n=3$ sectoral harmonic surface as described in the Introduction, with $\epsilon=0.1$. Gray curves mark the lines of symmetry and are included to bring out the shape of the surface. The dark contours straddling the equator mark the bifurcation sets: points outside and inside these contours have conjugate loci with 6 and 8 cusps respectively. The base points leading to the conjugate loci of Figure \ref{cpvarious} are labeled.} \label{regions}
\end{figure}
% data for cl sectoral 3 finer

%\end{appendices}

\bibliographystyle{plain}

\section*{\refname}

\bibliography{conjlocbib}

\begin{thebibliography}{10}

\bibitem{AP}
A.~N. Aleksandrov and K.~A. Piragas.
\newblock Geodesic structure.
\newblock {\em Theoretical and Mathematical Physics}, 38(1):48--56, 1979.

\bibitem{bazanski}
S.~L. Ba\.{z}a\'{n}ski.
\newblock Kinematics of relative motion of test particles in general
  relativity.
\newblock {\em Ann. Inst. Henri Poincar\'{e}}, 27(2):115--144, 1977.

\bibitem{bloch}
A.~M. Bloch and P.~E. Crouch.
\newblock Optimal control and geodesic flows.
\newblock {\em Systems and Control Letters}, 28(2):65--72, 1996.

\bibitem{Bonnard4}
B.~Bonnard and J.-B. Caillau.
\newblock Optimality results in orbit transfer.
\newblock {\em C. R. Acad. Sci. Paris}, 1(345):319--324, 2007.

\bibitem{Bonnard3}
B.~Bonnard and J.-B. Caillau.
\newblock Geodesic flow of the averaged controlled {K}epler equation.
\newblock {\em Forum Mathematicum}, 21(5):797--814, 2009.

\bibitem{Bonnard1}
B.~Bonnard, J.-B. Caillau, and G.~Picot.
\newblock Geometric and numerical techniques in optimal control of two- and
  three-body problems.
\newblock {\em Commun. Inf. Syst.}, 10(4):239--278, 2010.

\bibitem{Bonnard5}
B.~Bonnard, J.-B. Caillau, R.~Sinclair, and M.~Tanaka.
\newblock Conjugate and cut loci of a two-spere of revolution with application
  to optimal control.
\newblock {\em Ann. Inst. Henri Poincar\'{e} (C)}, 26(4):1081--1098, 2009.

\bibitem{Bonnard2}
B.~Bonnard, O.~Cots, and L.~Jassionnesse.
\newblock Geometric and numerical techniques to compute conjugate and cut loci
  on {R}iemannian surfaces.
\newblock {\em Geometric Control Theory and Sub-Riemannian Geometry}, 5:53--72,
  2014.

\bibitem{bruce}
J.~W. Bruce and P.~J. Giblin.
\newblock {\em Curves and singularities}.
\newblock Cambridge University Press, 1992.

\bibitem{TWetds}
T.~Combot and T.~Waters.
\newblock Integrability conditions of geodesic flow on homogeneous {M}onge
  manifolds.
\newblock {\em Ergodic Theory and Dynamical Systems}, 35(1):111--127, 2015.

\bibitem{grav}
J.~Gravesen, S.~Markvorsen, R.~Sinclair, and M.~Tanaka.
\newblock The cut locus of a torus of revolution.
\newblock {\em Asian J. Math.}, 9(1):103--120, 2005.

\bibitem{Hodg}
D.~E. Hodgkinson.
\newblock A modified equation of geodesic deviation.
\newblock {\em General Relativity and Gravitation}, 3(4):351--375, 1972.

\bibitem{Itoh1}
J.-I. Itoh and K.~Kiyohara.
\newblock The cut loci and the conjugate loci on ellipsoids.
\newblock {\em Manuscripta Mathematica}, 114(2):247--264, 2004.

\bibitem{Itoh3}
J.-I. Itoh and K.~Kiyohara.
\newblock The cut loci on ellipsoids and certain {L}iouville manifolds.
\newblock {\em Asian J. Math.}, 14(2):257--290, 2010.

\bibitem{Itoh2}
J.-I. Itoh and K.~Kiyohara.
\newblock Cut loci and conjugate loci on {L}iouville surfaces.
\newblock {\em Manuscripta Mathematica}, 136(1-2):115--141, 2011.

\bibitem{jacobi}
C.~G.~J. Jacobi.
\newblock {\em Vorlesungen \"{u}ber dynamik}.
\newblock Gehalten an der Universit\"{a}t zu K\"{o}nigsberg im Wintersemester
  1842-1843 und nach einem von C. W. Borchart ausgearbeiteten hefte. hrsg. von
  A. Clebsch.

\bibitem{kuz}
Y.~A. Kuznetsov.
\newblock {\em Elements of applied bifurcation theory}.
\newblock Springer, 2010.

\bibitem{myers}
S.~B. Myers.
\newblock Connections between differential geometry and topology, {I}: simply
  connected surfaces.
\newblock {\em Duke Math. J.}, 1(3):376--391, 1935.

\bibitem{poincare}
H.~Poincar\'{e}.
\newblock Sur les lignes geod\'{e}siques des surfaces convexes.
\newblock {\em Trans. Am. Math. Soc.}, 17:237--274, 1905.

\bibitem{Sinclair1}
R.~Sinclair.
\newblock On the last geometric statement of {J}acobi.
\newblock {\em Experimental Mathematics}, 12(4):477--485, 2003.

\bibitem{Sinclair5}
R.~Sinclair and M.~Tanaka.
\newblock Loki: software for computing cut loci.
\newblock {\em Experimental Mathematics}, 11(1):1--25, 2002.

\bibitem{Sinclair4}
R.~Sinclair and M.~Tanaka.
\newblock A bound on the number of endpoints of the cut locus.
\newblock {\em LMS J. Comp. Math.}, 9:21--39, 2006.

\bibitem{Sinclair2}
R.~Sinclair and M.~Tanaka.
\newblock Jacobi's last geometric statement extends to a wider class of
  {L}iouville surfaces.
\newblock {\em Mathematics of Computation}, 75(256):1779--1808, 2006.

\bibitem{Sinclair3}
R.~Sinclair and M.~Tanaka.
\newblock The cut locus of a two-sphere of revolution and {T}oponogov's
  comparison theorem.
\newblock {\em Tohoku Math. J.}, 59:379--399, 2007.

\bibitem{Vines}
J.~Vines.
\newblock Geodesic deviation at higher orders via covariant bitensors.
\newblock {\em General Relativity and Gravitation}, 47(5), 2015.

\bibitem{wald}
R.~M. Wald.
\newblock {\em General Relativity}.
\newblock University of Chicago Press, 1984.

\bibitem{wangguo}
Z.~X. Wang and D.~R. Guo.
\newblock {\em Special Functions}.
\newblock World Scientific, 1989.

\bibitem{TWmonge}
T.~J. Waters.
\newblock Non-integrability of geodesic flow on certain algebraic surfaces.
\newblock {\em Physics Letters A}, 376(17):1442--1445, 2012.

\bibitem{TWspherical}
T.~J. Waters.
\newblock Regular and irregular geodesics on spherical harmonic surfaces.
\newblock {\em Physica D: Nonlinear Phenomena}, 241(5):543--552, 2012.

\bibitem{whitehead}
J.~H.~C. Whitehead.
\newblock On the covering of a complete space by the geodesics through a point.
\newblock {\em Ann. Mathematics}, 36(3):679--704, 1935.

\end{thebibliography}

\end{document}